\documentclass[10pt,reqno]{amsart}
\usepackage[dvipsnames]{xcolor}
\usepackage{multirow,bigdelim}
\usepackage{amsthm, amsmath, amssymb, amsfonts, amstext, hyperref, graphicx, tikz, enumitem, mathtools, centernot, cancel, wasysym}
\usepackage[top=1.25in, bottom=1.25in, left=1in, right=1in]{geometry}
\hypersetup{colorlinks,citecolor=red,filecolor=black,linkcolor=blue,urlcolor=blue}

\newcommand{\excise}[1]{}
\newtheorem{theorem}{Theorem}[section]
\newtheorem{lemma}[theorem]{Lemma}
\newtheorem{corollary}[theorem]{Corollary}
\newtheorem{prop}[theorem]{Proposition}

\theoremstyle{definition}
\newtheorem{definition}[theorem]{Definition}
\newtheorem{example}[theorem]{Example}
\newtheorem{remark}[theorem]{Remark}

\DeclareMathOperator{\PF}{\mathsf{PF}}
\DeclareMathOperator{\IPF}{\mathsf{IPF}}

\newcommand{\outcome}{\mathcal{O}} 
\newcommand{\bioutcome}{\bar{\mathcal{O}}} 
\newcommand{\defterm}[1]{\textbf{\boldmath#1\unboldmath}} 

\newcommand{\reach}{\unrhd_R}
\newcommand{\reachby}{\unlhd_R}
\newcommand{\notreach}{\mkern-1mu\not\mathrel{\mkern1mu\unrhd_R}\mkern1mu}
\newcommand{\rcov}{{\,\mathrlap{\gtrdot}\rhd}_R\,}
\newcommand{\rcovby}{{\,\mathrlap{\lessdot}\lhd}_R\,}

\newcommand{\Sym}{\mathfrak{S}}
\newcommand{\cc}{\mathbf{c}}
\newcommand{\cf}{\mathsf{c}}
\newcommand{\df}{\mathsf{d}}

\newcommand{\DF}{\mathsf{D}}
\newcommand{\inv}{\ell}
\newcommand{\Zz}{\mathbb{Z}}
\newcommand{\0}{\emptyset}
\newcommand{\x}{\times}

\begin{document}
\title{Interval Parking Functions}
\thanks{This work was completed in part at the 2019 Graduate Research Workshop in Combinatorics, which was supported in part by NSF grant \#1923238, NSA grant \#H98230-18-1-0017,
a generous award from the Combinatorics Foundation, and Simons Foundation Collaboration Grants \#426971 (to M.~Ferrara) and \#315347 (to J.~Martin).}
\author{Emma Colaric}
\address{Department of Mathematics, University of Kansas, Lawrence, KS 66045}
\email{ecolaric@ku.edu}
\author{Ryan DeMuse}
\address{Department of Mathematics, University of Denver, Denver, CO 80210}
\email{ryan.demuse@du.edu}
\author{Jeremy L.\ Martin}
\thanks{JLM was supported in part by Simons Foundation Collaboration Grant \#315347.}
\address{Department of Mathematics, University of Kansas, Lawrence, KS 66045}
\email{jlmartin@ku.edu}
\author{Mei Yin}
\thanks{MY was supported in part by the University of Denver's Faculty Research Fund.}
\address{Department of Mathematics, University of Denver, Denver, CO 80210}
\email{mei.yin@du.edu}
\date{\today}
\subjclass[2010]{
05A05, 
05E15, 
06D99, 
20F55} 
\keywords{Parking function, bubble-sort, permutation, Bruhat order, weak order, sorting order}

\begin{abstract}
Interval parking functions (IPFs) are a generalization of ordinary parking functions in which each car is willing to park only in a fixed interval of spaces.  Each interval parking function can be expressed as a pair $(a,b)$, where $a$ is a parking function and $b$ is a dual parking function.  We say that a pair of permutations $(x,y)$ is \emph{reachable} if there is an IPF $(a,b)$ such that $x,y$ are the outcomes of $a,b$, respectively, as parking functions.  Reachability is reflexive and antisymmetric, but not in general transitive.  We prove that its transitive closure, the \emph{pseudoreachability order}, is precisely the bubble-sorting order on the symmetric group $\Sym_n$, which can be expressed in terms of the normal form of a permutation in the sense of du~Cloux; in particular, it is isomorphic to the product of chains of lengths $2,\dots,n$.  It is thus seen to be a special case of Armstrong's sorting order, which lies between the Bruhat and (left) weak orders.  
\end{abstract}

\maketitle

\section{Introduction}

We begin by briefly recalling the theory of parking functions, introduced in various contexts in~\cite{KW,Pyke,Riordan}; see \cite{Yan} for a comprehensive survey.  Consider a parking lot with $n$ parking spots placed sequentially along a one-way street.  A line of $n$ cars enters the lot, one by one.  The $i^{th}$ car drives to its preferred spot $a(i)$ and parks there if possible; if the spot is already occupied then the car parks in the first available spot.  The list of preferences $a=(a(1),\dots,a(n))$ is called a \defterm{parking function} if all cars successfully park; in this case the \defterm{outcome} is the permutation $\outcome(a)=w=(w(1),\dots,w(n))$, where the $i^{th}$ car parks in spot $w(i)$.  It is well known that the number of parking functions for $n$ cars is $(n+1)^{n-1}$.  Parking functions are an established area of research in combinatorics, with connections to labeled trees, non-crossing partitions, the Shi arrangement, symmetric functions, and other topics.

In this paper, we study a generalization of parking functions in which the $i^{th}$ car is willing to park only in an interval $[a(i),b(i)]\subseteq\{1,\dots,n\}$.  If all cars can successfully park then we say that the pair $(a,b)=( (a(1),\dots,a(n)), (b(1),\dots,b(n)) )$ is an \defterm{interval parking function}, or IPF.  (If $b(i)=n$ for all $i$, then we recover the classical case described above.)  It is easy to show that there are $n!(n+1)^{n-1}$ IPFs for $n$ cars, and that if $(a,b)$ is an IPF then the sequences $a$ and $b^*=(n+1-b(n),\dots,n+1-b(1))$ must both be parking functions, raising the question of the relationship between the permutations $\outcome(a)$ and $\outcome(b^*)$.

We say that a pair of permutations $(x,y)\in\Sym_n\x\Sym_n$ is \defterm{reachable}, written $x\reach y$, if there exists an IPF $(a,b)$ such that $x=\outcome(a)$ and $y^*=\outcome(b^*)$.  Reachability is \emph{not} a partial order on $\Sym_n$ because it is not transitive; however, its transitive closure is a partial order, which we call \defterm{pseudoreachability}.  The main result of this paper is that pseudoreachability order on~$\Sym_n$ is precisely the \emph{bubble-sorting order} on $\Sym_n$ (see \cite[Example 3.4.3]{BB}), which in turn is an instance of the more general \defterm{sorting order} defined by Armstrong~\cite{Armstrong} for Coxeter systems.   In particular, pseudoreachability lies between Bruhat and (left) weak order in $\Sym_n$, and it is a self-dual distributive lattice, poset-isomorphic to the product $C_2\x\cdots\x C_n$, where $C_i$ denotes the chain with $i$ elements.

The proof proceeds as follows.  The first significant result, Theorem~\ref{thm:bruhat}, states that $(x,y)$ is reachable only if $x\geq_By$, where $\geq_B$ denotes Bruhat order.  By counting the fibers of the map $(a,b)\mapsto(x,y)$, we establish Theorem~\ref{thm:RC}, the Reachability Criterion, which is a key technical tool in what follows.  Using this criterion, we show in \S\ref{sec:pseudoreach-order} that pseudoreachability is no weaker than left weak order, and use this result to show that it is graded by length, just like the Bruhat and weak orders.  This grading is key for the proof in \S\ref{sorting} that pseudoreachability coincides with the bubble-sorting order.

Initially, we had hoped to characterize reachability of a pair $(x,y)$ in terms of pattern-avoidance conditions on $x$ and $y$.  This does not appear to be possible in general, but Section~\ref{sec:avoid} contains partial results in this direction: Theorems~\ref{thm:213-avoiding} and~\ref{thm:x} give sufficient conditions for a pair $(x,y)$ to be reachable, provided that $x\geq_By$.  

The authors thank Margaret Bayer for proposing the study of interval parking functions to EC and JLM at the KU Combinatorics Seminar in the spring of 2019.  We are grateful to the Graduate Research Workshop in Combinatorics (GRWC) for providing the platform for this collaboration in 2019, and in particular we acknowledge helpful discussions with GRWC participants Sean English and Sam Spiro. We thank Bridget Tenner for her observant comments and Richard Stanley for his communications and suggestions for several directions of future investigation. 

\section{Preliminaries} \label{sec:notation}

Square brackets always denote integer intervals: For $m,n\in\Zz$ we put $[m,n]=\{m,\,\dots,\,n\}$ and $[n]=[1,n]$.
Lists of positive integers (including permutations) will be regarded as functions: thus we will write $a=(a(1),\dots,a(n))$ rather than $a=(a_1,\dots,a_n)$.
Thus notation such as $x[a,b]$ means $\{x(a),x(a+1),\dots,x(b)\}$.
To simplify notation, we sometimes drop the parentheses and commas: e.g., $2431=(2,4,3,1)$.

Let $a=(a(1),\dots,a(n))$ and $b=(b(1),\dots,b(n))\in\Zz^n$.  We write $a\leq_Cb$ if $a(i)\leq b(i)$ for all $i\in[n]$; this is the \defterm{componentwise partial order} on $\Zz^n$.
The \defterm{conjugate} (or reverse complement) of $x\in[n]^n$ is the vector $x^* = (n+1-x(n), \dots, n+1-x(1))$.  Conjugation is an involution that reverses componentwise order. 

If $\geq$ is a partial ordering on a set $S$, then $\gtrdot$ denotes the corresponding covering relation: $x\gtrdot y$ if $x>y$ and there exists no $z$ such that $x>z>y$.  It is elementary that if $\geq_1$ is a partial order at least as strong as $\geq_2$ (i.e., $x\geq_2y$ implies $x\geq_1y$), then $x>_2y$ and $x\gtrdot_1y$ together imply $x\gtrdot_2y$.

The symmetric group of all permutations of $[n]$ is denoted by~$\Sym_n$.  We will as far as possible follow the notation and terminology for the symmetric group used in \cite{BB}.  We set $e=(1,\dots,n)$ (the identity permutation) and $w_0=(n,n-1,\dots,1)$.
The permutation transposing $i$ and $j$ and fixing all other values is denoted $t_{ij}$, and we set $s_i=t_{i,i+1}$; the elements $s_1,\dots,s_{n-1}$ are the \defterm{standard generators}.  Our convention for multiplication is right to left, which is consistent with treating permutations as bijective functions $[n]\to[n]$.  Thus $t_{ij}x$ is obtained by transposing the \textit{digits} $i,j$ wherever they appear in $x$, while $x t_{ij}$ is obtained by transposing the digits in the $i^{th}$ and $j^{th}$ \textit{positions}.

We list some standard facts from the theory of $\Sym_n$ as a Coxeter system of type~A, with generators $S=\{s_1,\dots,s_{n-1}\}$; see \cite{BB} for details.  The \defterm{length} $\ell(x)$ of $x\in\Sym_n$ is the smallest number $k$ such that $x$ can be written as a product $s_{i_1}\cdots s_{i_k}$ of standard generators; in this case $s_{i_1}\cdots s_{i_k}$ is called a \defterm{reduced word} for $x$.  It is a standard fact that length equals number of inversions:
\begin{equation} \label{length-inv}
\ell(x)=\{(i,j):\ 1\leq i<j\leq n,\ x(i)>x(j)\}.
\end{equation}

The \defterm{Bruhat order} is the partial order $>_B$ on $\Sym_n$ defined as the transitive closure of the relations $x>t_{ij}x$ whenever $\inv(x)>\inv(t_{ij}x)$.  (Multiplying $x$ by $t_{ij}$ on the right rather than the left produces the same order, because $x t_{ij} x^{-1}$ is a transposition and $x t_{ij}=(x t_{ij} x^{-1})x$.)
The \defterm{(left) weak order} $>_W$ is the transitive closure of the relations $x>s_ix$ whenever $s$ is a standard generator and $\inv(x)>\inv(sx)$.  Both of these orders make $\Sym_n$ into a graded poset with bottom element $e$ and top element $w_0$.

\section{Parking functions and interval parking functions} \label{sec:intro}

We begin by recalling the theory of parking functions, introduced in various contexts in~\cite{KW,Pyke,Riordan}; see \cite{Yan} for a comprehensive survey.
Let $a=(a(1),\,\dots,\,a(n))\in[n]^n$.  Consider a parking lot with $n$ parking spaces placed sequentially along a one-way street.  Cars 1,\,\dots,\,$n$ enter the lot in order and try to park.

{\bf Algorithm~A:} The $i^{th}$ car parks in the first available space in the range $[a(i),n]$.  If no space in the range $[a(i),n]$ is available, the algorithm fails.

If Algorithm~A succeeds in parking every car, then the preference vector $a$ is called a \defterm{parking function}.  The set of all parking functions $a=(a(1),\,\dots,\,a(n))$ is denoted $\PF_{n}$.  It is well known that $|\PF_n|=(n+1)^{n-1}$ and that
\[\PF_n = \{a\in[n]^n:\ \tilde a(i)\leq i\ \ \forall i\}\]
where $\tilde{a}$ is the unique non-decreasing rearrangement of $a$; in particular, every rearrangement of a parking function is a parking function.

The \defterm{outcome} of a parking function $a\in\PF_n$ is the permutation $x=\outcome(a)=(x(1),\,\dots,\,x(n))$, where $x(i)$ is the spot in which car $i$ parks given the preference list $a$.

We now modify Algorithm~A to obtain our central object of study.

{\bf Algorithm~B:} Let $a,b\in[n]^n$ with $a\leq_Cb$.  The $i^{th}$ car parks in the first available space in the range $[a(i),b(i)]$.  If no space in the range $[a(i),b(i)]$ is available, the algorithm fails.

\begin{definition}
If Algorithm~B succeeds in parking every car, then $\cc=(a,b)$ is called an \defterm{interval parking function}, or IPF.  The set of all interval parking functions for $n$ cars is denoted $\IPF_n$.  The \defterm{feasible interval} for the $i^{th}$ car is $[a(i),b(i)]$.
\end{definition}

For example,
\[\IPF_2 = \{(11,12),\ (11,22),\ (12,12),\ (12,22),\ (21,21),\ (21,22)\}.\]
Unlike ordinary parking functions, IPFs are \emph{not} invariant under the action of $\Sym_2$ by permuting cars.  For example, $(11,12)$ is an IPF but $(11,21)$ is not.

\begin{prop} \label{IPF-characterization} Let $a,b\in[n]^n$.  Then:
\begin{enumerate}
\item $a\in\PF_n$ if and only if $(a,(n,\dots,n))\in\IPF_n$.
\item $(a,b)\in\IPF_n$ if and only if $a\in\PF_n$ and $\outcome(a) \le_C b$.
\end{enumerate}
\end{prop}
\begin{proof}
For (1), if $b(i)=n$ for all $i$ then Algorithm~B is identical to Algorithm~A.
For (2), if the given conditions hold, then the execution of Algorithm~B mimics that of Algorithm~A.  On the other hand, if $a$ is not a parking function, then some car will not find a spot, while if $\outcome(a)\not\leq_Cb$ then some car will not find a spot in its own feasible interval.
\end{proof}

As a consequence of the proof of (2), the outcome $\outcome(\cc)$ of $\cc=(a,b)$ is just $\outcome(a)$.  Moreover, for every $a\in\PF_n$, there are precisely $n!$ choices for $b$ such that $(a,b)\in\IPF_n$.  (This fact was first observed by Sean English.) In particular,
\begin{equation} \label{count-IPF}
\left|\IPF_{n}\right| = n!(n+1)^{n-1}.
\end{equation}

\begin{prop} \label{lots-of-facts}
Let $\cc = (a,b)\in\IPF_n$.  Then:
\begin{enumerate}
\item $b^*\in\PF_n$.
\item $a \le_C \outcome(\cc) \le_C b$ and $\outcome(b^*)^* \le_C b$.
\end{enumerate}
\end{prop}

\begin{proof}
\noindent
\begin{enumerate}
\item From $\outcome(\cc) \leq_C b$, one has $b^* \leq_C \outcome(\cc)^*$, the latter is a permutation. Hence $b^*$ is a parking function.
\item Evidently $a \le_C \outcome(\cc) \le_C b$. By (1), $b^*$ is a parking function. Thus $b^* \le_C \outcome(b^*)$. Conjugation reverses the order $\leq_C$ and is an involution, so $\outcome(b^*)^* \le_C (b^*)^* = b$.
\qedhere
\end{enumerate}
\end{proof}

\section{The Bruhat property} \label{sec:Bruhat}

In this section, we prove another property of interval parking functions related to Bruhat order on permutations.  We use the following characterization of Bruhat order~\cite[Thm.~2.1.5, p.32]{BB}: $y\leq_Bx$ if and only if
\begin{equation} \label{bruhat-criterion}
y\langle i,j\rangle\leq x\langle i,j\rangle \qquad \forall i,j\in[n]
\end{equation}
where
\begin{equation} \label{angle-brackets}
u\langle i,j\rangle = \#\{k\in[i]:\ u(k)\geq j\}.
\end{equation}
(This quantity is notated $u[i,j]$ in \cite{BB}, but we reserve that notation for the image of an interval under a permutation.)
For later use, we observe that by pigeonhole, it is always the case that
\begin{equation} \label{bracket-ineq}
x\langle i,j\rangle\geq i-j+1.
\end{equation}
Suppose that $\cc=(a,b)$ is an IPF, and let $x=\outcome(a)$ and $y=\outcome(b^*)^*$.  Then $x\langle i,j\rangle$ is the number of cars $1,\,\dots,\,i$ that park at or after spot $j$ under the parking function $a$.

\begin{theorem} \label{thm:bruhat}
Suppose that $\cc = (a,b)$ is an IPF.  Let $x=\outcome(a)$ and $y=\outcome(b^*)^*$.  Then $x\geq_By$.
\end{theorem}

\begin{proof}
First, we may assume without loss of generality that $x=a$, because replacing $a$ with $x$ doesn't change the execution of Algorithm~B (the $i^{th}$ car will have to drive to spot $x(i)$ anyway, and it is able to park there because $\cc$ is an IPF).

Fix $i,j\in[n]$, and let $p=x\langle i,j\rangle$ and $q=y\langle i,j\rangle$.  By~\eqref{bruhat-criterion} we wish to show that $p\geq q$.  By definition of $y\langle i,j\rangle$ we have
\begin{equation} \label{bruhat:1}
\Big|y[1,i]\cap[j,n]\Big|=q
\end{equation}
or equivalently
\begin{equation} \label{bruhat:2}
\Big|y^*[n-i+1,n]\cap[1,n+1-j]\Big|=q.
\end{equation}
Therefore, when Algorithm~A is run on the parking function $b^*$ with outcome $y^*$, the first $n-i$ cars must leave open at least $q$ spaces in the range $[1,n+1-j]$, so they cannot fill as many as $(n+1-j)-q+1=n-j-q+2$ of them.  Therefore, $b^*[1,n-i]$ can contain no subset
$\{v(1),\,\dots,\,v^*(n-j-q+2)\}$ such that
\[(v(1),\,\dots,\,v^*(n-j-q+2))\leq_C (q,\,\dots,\,n+1-j).\]
Equivalently, $\{b(i+1),\,\dots,\,b(n)\}$ can contain no subset
$\{v(1),\,\dots,\,v(n-j-q+2)\}$ such that
\[(v(1),\,\dots,\,v(n-j-q+2))\geq_C (j,\,\dots,\,n-q+1).\]
It follows that when Algorithm~B is run on $\cc$, no more than $n-j-q+1$ of the last $n-i$ cars will park in the spots $[j,n]$.  On the other hand, since $x=\outcome(\cc)$, no more than $p=x\langle i,j\rangle$ of the first $i$ cars can park in the spots $[j,n]$.  Therefore, the total number of cars that park in $[j,n]$ is at most
\[(n+1-j-q)+p = |[j,n]|+(p-q).\]
On the other hand, exactly $|[j,n]|$ cars park in $[j,n]$.  It follows that $p\geq q$, as desired.
\end{proof}

Theorem~\ref{thm:bruhat} asserts that there is a well-defined \defterm{bioutcome} function
\begin{equation} \label{define-bioutcome}
\begin{array}{llll}
\bioutcome:&\IPF_n&\to&\{(x,y)\in\Sym_n\x\Sym_n:\ x\geq_By\}\\
&(a,b)&\mapsto&(\outcome(a),\outcome(b^*)^*).
\end{array}
\end{equation}
We say that a pair $(x,y)\in\Sym_n\x\Sym_n$ is \defterm{reachable} if it is in the image of $\bioutcome$; in this case we write $x\reach y$.  (We use this notation rather than $x\geq_R y$ because reachability is not a partial order on $\Sym_n$, as we will discuss shortly.)
Then Theorem~\ref{thm:bruhat} asserts that all reachable pairs are related in Bruhat order.

\begin{remark} \label{unreachable}
If $a$ and $b^*$ are parking functions such that $\outcome(a)\geq_B\outcome(b^*)^*$, it does \emph{not} follow that $\cc=(a,b)$ is an IPF.  For example, if $a=w_0$ and $b$ is a permutation, then certainly $a=\outcome(a)\geq_B\outcome(b^*)^*=b$, but $(a,b)$ is an IPF only if $b=w_0$ as well.

Moreover, if $x,y\in\Sym_n$ with $x\geq_By$, there does not necessarily exist any IPF $\cc=(a,b)$ such that $\bioutcome(\cc)=(x,y)$.  For example, when $n=3$, take $(x,y)=(321,213)$, so that $y^*=132$.  Then $a=321$ is the only parking function with $\outcome(a)=x$.  By Prop.~\ref{lots-of-facts}(2) we must have $b\geq_Ca$, so $b\in\{321,331, 322, 332, 323,333\}$ and $b^*\in\{321,311,221,211,121,111\}$.  But none of these parking functions have outcome $y^*=132$.
\end{remark}

The relation of reachability is reflexive (because $\bioutcome(x,x)=(x,x)$ for all $x\in\Sym_n$) and antisymmetric (as a consequence of Theorem~\ref{thm:bruhat}).  However, it is not transitive: for example,
$321\notreach 213$, as just shown, but $(321,312)=\bioutcome(312,322)$ and $(312,213)=\bioutcome(312,313)$ are reachable.  This observation motivates the following definition.

\begin{definition} \label{def-pseu}
We say that $(x,y)$ is \defterm{pseudoreachable}, written $x\geq_P y$, if there is a sequence $x=x_0\reach x_1\reach\cdots\reach x_k=y$.  That is, pseudoreachability is the transitive closure of reachability.  As such, it is a partial order on $\Sym_n$, which by Theorem~\ref{thm:bruhat} is no stronger than Bruhat order.
\end{definition}

For reference, we summarize the various order-like relations that we will consider.
\medskip
\begin{center} {\renewcommand\arraystretch{1.2}
\begin{tabular}{lll} \hline
$a\geq_C b$ & Componentwise order & on $\Zz^n$\\
$x\geq_B y$ & Bruhat order &  \rdelim\}{4}{1em}[\ on $\Sym_n$] \\
$x\geq_W y$ & Left weak order\\
$x\reach y$ & Reachability (not transitive)\\
$x\geq_P y$ & Pseudoreachability\\ \hline
\end{tabular} }
\end{center}
\medskip

\section{Reachability via counting fibers of the bioutcome map} \label{sec:reachable}

Fix a pair of permutations $(x,y)\in\Sym_n\x\Sym_n$. How can we determine if $(x,y)$ is reachable?  More generally, what is the number $\phi(x,y)=|\bioutcome^{-1}(x,y)|$ of IPFs $(a,b)$ with bioutcome~$(x,y)$?

We can answer this enumerative question quickly, although the resulting formula is recursive and somewhat opaque.  First, for each $i$, the number of possibilities $\cf_i=\cf_i(x,y)$ for $a(i)$ is the size of the largest block of spaces ending in $x(i)$ that are all occupied by one of the first $i$ cars.
That is,
\[\cf_i=\cf_i(x,y) = \max\left\{ j \in [1, x(i)]: x^{-1}(x(i)-k) \le i \text{ for all } 0 \le k \le j-1 \right\}.\]

Second, given $a(1),\dots,a(i)$, the number of possibilities for $b(i)$ is $\df_i=\df_i(x,y) = \#\DF_i(x,y)$, where
\[\DF_i(x,y)=\{k\in[0,J_i-1]:\ y(i) + k \ge x(i)\}\]
and
\[J_i =\max\{j\in[1,n+1-y(i)]:\ y^{-1}(y(i)+s) \ge i \text{ for all }0 \le s \le j-1\}.\]

The definition of $J_i$ is analogous to that of $\cf_i$: it is the size of the largest block of spaces ending in $n+1-y(i)$ that are all occupied by one of the first $n+1-i$ cars, so it is the number of possible values for $b^*_i$ under which $\outcome(b^*)=y^*$.
The additional condition $y(i)+k\geq x(i)$ in the definition of $\DF_i$ ensures that $(a,b)$ is an IPF because the upper bound on $x(i)$ given by $b(i)$ does not conflict with where the $i^{th}$ car parks under Algorithm~B.

The sequences $\cf=(\cf_1,\dots,\cf_n)$ and $\df=(\df_1,\dots,\df_n)$
then determine the size of the fibers of $\bioutcome$:
\begin{equation}
\phi(x,y)=\left|\bioutcome^{-1}(x,y)\right| = \prod_{i=1}^{n} \cf_i\df_i.
\end{equation}

\begin{example}
Let $x = 361245$ and $y = 341256$. Then $\cf=(1,1,1,2,4,5)$ and $\df=(4,1,2,1,2,1)$, so there are $2^34^25^1=640$ IPF's with bioutcome $(x,y$).
\end{example}

It is clear from the definition that $1\leq\cf_i\leq i$ for all $i$.  On the other hand, one or more $\df_i$ may be zero.  The pair $(x,y)$ is reachable if and only if $\df_i>0$ for all $i$; we refer to this as the \textbf{Count Criterion} for reachability.

Evidently, the largest fiber occurs when $x$ and $y$ both equal the identity permutation in $\Sym_n$.  In this case $\cf=(1,2,\dots,n)$ and $\df=(n,n-1,\dots,1)$, and the fiber size is $(n!)^2$.  At the opposite end of the spectrum, if $x=y=(n,\dots,1)$, then $\phi(x,y) = 1$.

Perhaps a better way to think about reachability is the following criterion. If we are solely interested in reachability and not the number of IPFs that achieve a given outcome, we can rephrase reachability more directly in terms of the permutations $x$ and $y$.

\begin{theorem}[\textbf{Reachability Criterion}]\label{thm:RC}
Let $x,y\in\Sym_n$.  Then
\begin{equation}\label{RC}
x\reach y \quad\iff\quad [y(i),x(i)] \subseteq y[i,n] \quad \forall\, i \in [n].\tag{\textbf{RC}}
\end{equation}
\end{theorem}

\begin{proof}
Let $i\in[n]$.  We will show that $\df_i(x,y)>0$ if and only if $[y(i),x(i)] \subseteq y[i,n]$.

Suppose that $[y(i),x(i)] \setminus y[i,n]\neq \0$.  That is, there is some $m\in[y(i),x(i)]$ such that $y^{-1}(m)<i$. Thus $J_i \leq m-y(i)$, so $y(i)+k<m\leq x(i)$ for all $k<J_i$, so $\df_i(x,y)=0$.

Now assume that $[y(i),x(i)] \subseteq y[i,n]$. We wish to show that $\DF_i\neq\0$. If $y(i)\geq x(i)$, then $0\in\DF_i$.  On the other hand, if $y(i)<x(i)$, then $m=x(i)-y(i)>0$, and for all $0 \leq k \leq m$ we have $y^{-1}(y(i)+k) \ge i$.  Therefore $J_i > m$ and $m\in\DF_i$.
\end{proof}

It is worth emphasizing that the Reachability Criterion is sufficient, but not necessary, for showing that $x\geq_Py$.  For example, the pair $(x,y)=(321,213)$ fails~\eqref{RC} for $i=2$, but nonetheless $x\geq_Py$.
\begin{prop} \label{cf-df-facts}
The sequence $\df(x,y)$ has the following properties.
\begin{enumerate}[label=(\alph{enumi})]
\item\label{dfone} $\df_1\geq1$.
\item\label{dfi} For each $i$, if $y(i)\geq x(i)$, then $\df_i\geq1$.
\item\label{dfn} If $x\geq_By$, then $\df_n=1$.
\end{enumerate}
\end{prop}

\begin{proof}  The first two assertions are direct consequences of~\eqref{RC}.  For~\ref{dfone}, we have $[y(1),x(1)] \subseteq [n] = y[n]$, and for~\ref{dfi}, if $y(i)\geq x(i)$ then $[y(i),x(i)]\subseteq\{y(i)\}\subseteq y[i,n]$.

For~\ref{dfn}, if $y\leq_B x$, then $y(n) \ge x(n)$ (a consequence of the inequalities~\eqref{bruhat-criterion} for $i=n-1$ and all $j$), so $\df_n>0$ by part~\ref{dfi}.  Observe that
\[J_n =\max\{j:\ y(n)+k\leq n \text{ and } y^{-1}(y(n)+k) \ge n \text{ for all }0 \le k \le j-1\} = 1\]
because the conditions are true for $k=0$ but false for $k>0$.  Therefore, $\DF_n=\{k\in[0,0]:\ y(n)\geq x(n)\}=\{0\}$ and $\df_n=\#\DF_n=1$.
\end{proof}

\section{Pseudoreachability order is graded} \label{sec:pseudoreach-order}

In this section, we prove that the pseudoreachability order $\geq_P$ on $\Sym_n$ is graded by length, just like the Bruhat and weak orders.

Temporarily, we will use the notation $x\rcov y$ to mean that $x\reach y$ and $\inv(x) = \inv(y) + 1$.  Note that if $x\rcov y$ then $x\gtrdot_Py$ (because $x\gtrdot_By$).  Our goal is to prove the converse of the last statement, which will imply that pseudoreachability is graded by length.

We have already shown that pseudoreachability order is no stronger than Bruhat order $\geq_B$.  We next show that it is no weaker than left weak order $\geq_W$.

\begin{prop} \label{lem:inv-1}
If $x\gtrdot_Wy$, then $x\rcov y$.
\end{prop}

\begin{proof}
Suppose that $x\gtrdot_Wy$, i.e., that $x=s_ay$, where $j=y^{-1}(a)<y^{-1}(a+1)=k$.  Then Prop.~\ref{cf-df-facts}\ref{dfi} implies that $\df_i(x,y)>0$ for all $i\in[n]\setminus\{j\}$.  Meanwhile $[y(j),x(j)]=\{a,a+1\}=\{y(j),y(k)\}\subseteq[y(j),y(n)]$, so~\eqref{RC} implies that $\df_j(x,y)>0$ as well.
\end{proof}

For each $x\in\Sym_n$, let $\hat x$ be the permutation in $\Sym_{n-1}$ defined by
\begin{equation} \label{hats-on}
\hat x(i) = \begin{cases} x(i) & \text{ if } x(i)<x(n),\\ x(i)-1 & \text{ if } x(i)>x(n).\end{cases}
\end{equation}

\begin{lemma} \label{lem:proj}
Let $x,y \in \Sym_n$ with $x(n) = y(n)$.
Then $x\reach y$ if and only if $\hat x\reach\hat y$.
\end{lemma}

\begin{proof}
By~\eqref{RC}, the proof reduces to showing that
\begin{subequations}
\begin{equation} \label{reach-xy}
[y(i),x(i)] \subseteq y[i,n] \qquad \forall i\in[n]
\end{equation}
if and only if
\begin{equation} \label{reach-proj}
[\hat y(i),\hat x(i)] \subseteq \hat y[i,n] \qquad \forall i\in[n-1].
\end{equation}
\end{subequations}

($\implies$) Assume that~\eqref{reach-xy} holds.  Let $i\in[n-1]$ and $a \in [\hat y(i),\hat x(i)]$.  There are two cases to consider.

\textit{Case 1a}:  $a < y(n)$. Then 
$\hat y(i)\leq a<y(n)$, so $\hat y(i)=y(i)$ (since~\eqref{hats-on} implies that if $\hat y(i)=y(i)-1$ then $\hat y(i)\geq y(n)$).
Thus
\[ [\hat y(i),a] = [y(i),a] \subseteq [y(i),x(i)] \subseteq y[i,n] \]
because $a\leq\hat x(i)\leq x(i)$, and by~\eqref{reach-xy}.
Therefore $a = y(k)=\hat y(k)$ for some $k\in[i,n-1]$.

\textit{Case 1b}: $a \geq y(n)$. Then, since $\hat y(i) \geq y(i) - 1$ and $x(i) \geq \hat x(i) \geq y(n)$, $a \in [\hat y(i),\hat x(i)]$ implies that $a \in [y(i)-1,x(i)-1]$, i.e., $y(i) \leq a+1 \leq x(i)$.
By~\eqref{reach-xy} there is some $k\in[i,n]$ such that $a+1 = y(k)$.  In fact $k\neq n$ (since $a+1>y(n)$), so $\hat y(k) = y(k) - 1 = a$ and so $a \in \hat y[i,n-1]$.

In both cases we have proved~\eqref{reach-proj}.
\medskip

($\impliedby$) Assume that~\eqref{reach-proj} holds.  It is immediate that~\eqref{reach-xy} holds when $i=n$, so fix $i\in[n-1]$ and $a \in [y(i),x(i)]$.  We wish to show that $a=y(k)$ for some $k\in[i,n]$.  This is clear if $a=y(n)$, so assume $a\neq y(n)$.

\textit{Case 2a}:  $a < y(n)$. Since $a\in[y(i),x(i)]$, either $a = x(i)$ or $a < x(i)$. If $a = x(i)$, then $a = x(i) = \hat x(i)$. If $a < x(i)$, then $a \leq \hat x(i)$ since $\hat x(i) \geq x(i) - 1$. In either case,
\[ [y(i),a] = [\hat y(i),a] \subseteq [\hat y(i),\hat x(i)] \subseteq \hat y[i,n-1]. \]
Thus $a=\hat y(k)=y(k)$ for some $k\in[i,n-1]$.

\textit{Case 2b}: $a > y(n)$.  Since $a\in[y(i),x(i)]$, either $a = y(i)$ or $a > y(i)$. If $a = y(i)$, then $a - 1 = y(i) - 1 = \hat y(i)$ since $y(i) > y(n)$. If $a > y(i)$, then we know that $a-1 \geq \hat y(i)$ since $y(i) \geq \hat y(i)$. It follows that $a-1\in[\hat y(i),\hat x(i)]$, so, by~\eqref{reach-proj}, there is some $k\in[i,n-1]$ such that $a-1 = \hat y(k)\geq y(n)$. Therefore, $a = y(k)$.

In both cases we have proved~\eqref{reach-xy}.
\end{proof}

\begin{corollary} \label{cor:rcover}
Let $x,y \in \Sym_n$ with $x(n) = y(n)$.
Then $x\rcov y$ if and only if $\hat x\rcov\hat y$.
\end{corollary}
\begin{proof}
The definition of $\hat x$ implies that
\begin{equation} \label{hat-inv}
\inv(\hat x) = \inv(x)-(n-x(n)),
\end{equation}
which together with Lemma~\ref{lem:proj} produces the desired result.
\end{proof}

\begin{prop} \label{reachable-graded}
Let $x,y\in\Sym_n$ such that $x\reach y$, and let $m=\inv(x)-\inv(y)$.  Then there exists a chain
\begin{equation} \label{desired-chain}
x_0=y\rcovby x_1\rcovby\cdots\rcovby x_m=x.
\end{equation}
\end{prop}

\begin{proof}
The proof proceeds by double induction on $n$ and $m$.  The conclusion is trivial when $n\leq2$ or $m\leq1$.  Accordingly, let $n>2$ and $m>1$, and assume inductively that the theorem holds for all $(n',m')<_C(n,m)$.

First, suppose that $x(n) = y(n)$.  Then $\hat{x}\reach\hat{y}$ by Lemma~\ref{lem:proj} where $\hat x,\hat y$ are defined by~\eqref{hats-on}.  Moreover,
$\inv(\hat{x}) - \inv(\hat{y}) = \inv(x) - \inv(y) = m$ by~\eqref{hat-inv}. Therefore, by the induction hypothesis, there is a chain
$\hat{y} = \hat x_0 \rcovby \hat x_1 \rcovby \cdots \rcovby \hat x_m = \hat{x}$
in $\Sym_{n-1}$,
which by Corollary~\ref{cor:rcover} can be lifted to a chain of the form~\eqref{desired-chain}.

Second, suppose that $x(n) \neq y(n)$. Since $x\geq_By$ by Theorem \ref{thm:bruhat}, in fact $x(n) < y(n)$ (as noted in the proof of Prop.~\ref{cf-df-facts}\ref{dfn}).
Let $p=y(n)-1$; then $p\in[1,n-1]$, so we may set $q=y^{-1}(p)$ and $z=s_py=y t_{q,n}$.
Then $z\gtrdot_Wy$ and so $z\rcov y$ by Prop.~\ref{lem:inv-1}.  We will show that $x\reach z$ using~\eqref{RC}.

\textit{Case 1}: $1\leq i\leq q$.  Then $[z(i),x(i)]\subseteq[y[i],x(i)]$ and $y[i,n]=z[i,n]$, so $\df_n(x,y)\geq1$ implies $\df_n(x,z)\geq1$.

\textit{Case 2}: $q<i<n$.
Then $p=y(q)\not\in y[i,n]$, so by~\eqref{RC} $p\not\in[y(i),x(i)]$.  Thus $p+1\not\in[y(i)+1,x(i)+1]$,
and certainly $p+1=y(n)\neq y(i)$.  Thus $[y(i),x(i)] \subseteq y[i,n]\setminus\{y(n)\}=y[i,n-1]$ and
\[[z(i),x(i)] = [y(i),x(i)] \subseteq y[i,n-1]=z[i,n-1]\subseteq z[i,n]\]
so again $\df_n(x,z)\geq1$.

\textit{Case 3}: $i=n$.  Then $x(n) \leq  y(n)-1 =z(n)$, so $\df_{n}(x,z) \geq 1$ by Prop.~\ref{cf-df-facts}\ref{dfi}.

Taken together, the three cases imply $x\reach z$.  By induction there is a chain $x_1=z\rcovby\cdots\rcovby x_m=x$, and appending $x_0=y$ produces a chain of the form~\eqref{desired-chain}.
\end{proof}

\begin{theorem} \label{pseudoreachable-graded}
Pseudoreachability order is graded by length.
\end{theorem}
\begin{proof}
The definition of pseudoreachability as the transitive closure of reachability order implies that if $x_0<_P\cdots<_Px_m$ is a maximal chain, then in fact each $x_{i-1}\reachby x_i$ for all $i$.  Now, maximality together with Prop.~\ref{reachable-graded} implies in turn that in fact $x_{i-1}\rcovby x_i$.
\end{proof}

For comparison, the Hasse diagrams of Bruhat, pseudoreachability, and left weak orders on $\Sym_3$ are shown in Figure~\ref{fig:s3}, together with the
reachability relation (which is reflexive and antisymmetric, but not transitive).  The three partial orders on $\Sym_4$  are shown in Figure~\ref{fig:s4}.

\begin{figure}
\begin{center}
\begin{tikzpicture}
\newcommand{\hspacing}{4.5}
\begin{scope}[shift={(0,0)}]
	\draw[black] (0,0)--(-1,1)--(-1,2)--(0,3)--(1,2)--(1,1)--cycle;
	\draw[black] (-1,1)--(1,2) (-1,2)--(1,1);
	\foreach \x/\y/\w in {0/0/123, -1/1/132, 1/1/213, -1/2/231, 1/2/312, 0/3/321} \node[fill=white] at (\x,\y) {\sf\w};
	\node at (0,-.5) {Bruhat order $\geq_B$};
\end{scope}
\begin{scope}[shift={(\hspacing,0)}]
	\draw[black] (0,0)--(-1,1)--(-1,2)--(0,3)--(1,2)--(1,1)--cycle;
	\draw[black] (-1,1)--(1,2);
	\foreach \x/\y/\w in {0/0/123, -1/1/132, 1/1/213, -1/2/231, 1/2/312, 0/3/321} \node[fill=white] at (\x,\y) {\sf\w};
	\node at (0,-.5) {Pseudoreachability $\geq_P$};
\end{scope}
\begin{scope}[shift={(2*\hspacing,0)}]
	\draw[black] (0,0)--(-1,1)--(-1,2)--(0,3)--(1,2)--(1,1)--cycle;
	\foreach \x/\y/\w in {0/0/123, -1/1/132, 1/1/213, -1/2/231, 1/2/312, 0/3/321} \node[fill=white] at (\x,\y) {\sf\w};
	\node at (0,-.5) {Left weak order $\geq_W$};
\end{scope}
\begin{scope}[shift={(3*\hspacing,0)}]
	\foreach \p in {(-1,1), (1,1), (-1,2), (1,2), (0,3)} \draw[black] (0,0)--\p;
	\foreach \p in {(-1,2), (1,2), (0,3)} \draw[black] (-1,1)--\p;
	\draw[black] (1,1)--(1,2)--(0,3)--(-1,2);
	\foreach \x/\y/\w in {0/0/123, -1/1/132, 1/1/213, -1/2/231, 1/2/312, 0/3/321} \node[fill=white] at (\x,\y) {\sf\w};
	\node at (0,-.5) {Reachability $\reach$};
\end{scope}
\end{tikzpicture}
\caption{Bruhat, pseudoreachability, left weak order, and reachability on $\Sym_3$\label{fig:s3}}
\end{center}
\end{figure}
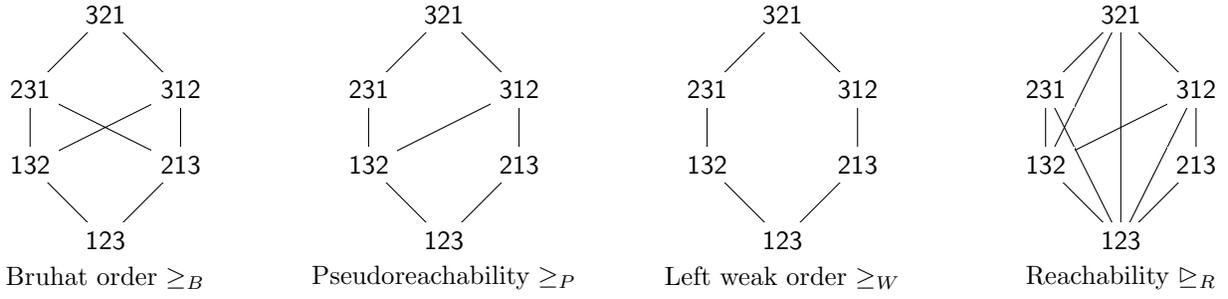

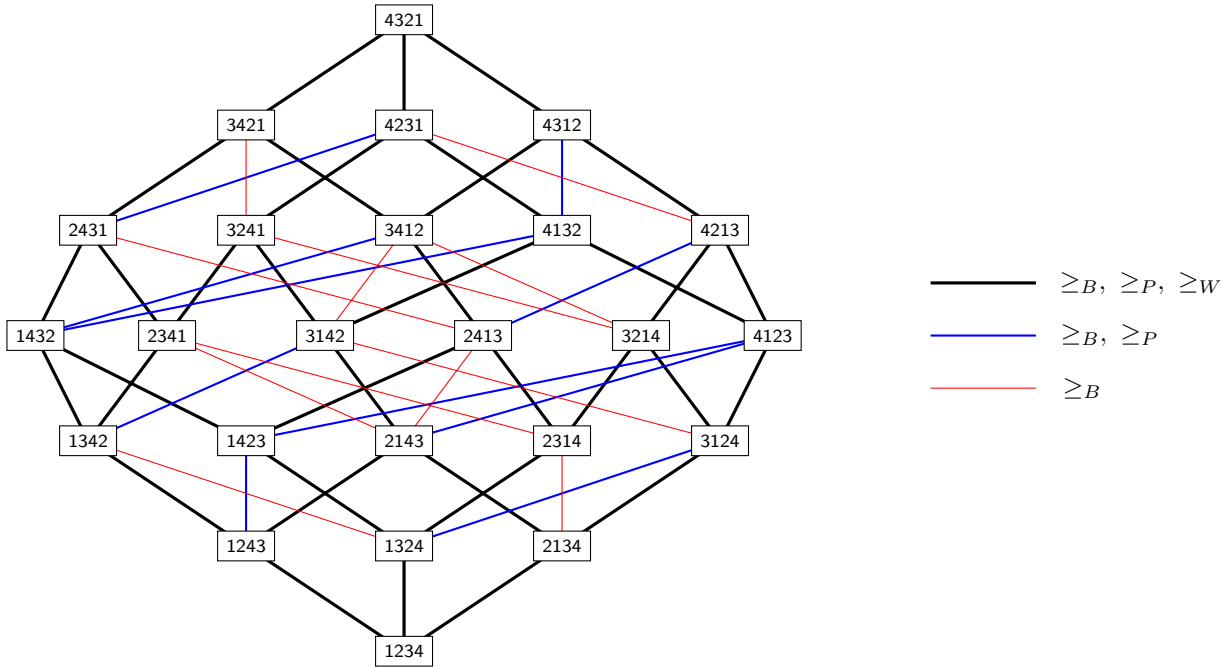
\begin{figure}
\begin{center}
\begin{tikzpicture}[scale=1.4]
															\coordinate (p4321) at (0,6);
									\coordinate (p3421) at (-1.5,5);	\coordinate (p4231) at (0,5);	\coordinate (p4312) at (1.5,5);
			\coordinate (p2431) at (-3,4);	\coordinate (p3241) at (-1.5,4);	\coordinate (p3412) at (0,4);	\coordinate (p4132) at (1.5,4);	\coordinate (p4213) at (3,4);
			\coordinate (p1432) at (-3.5,3);	\coordinate (p2341) at (-2.25,3); \coordinate (p3142) at (-.75,3);	\coordinate (p2413) at (.75,3);	\coordinate (p3214) at (2.25,3);	\coordinate (p4123) at (3.5,3);
			\coordinate (p1342) at (-3,2);	\coordinate (p1423) at (-1.5,2);	\coordinate (p2143) at (0,2);	\coordinate (p2314) at (1.5,2);	\coordinate (p3124) at (3,2);
									\coordinate (p1243) at (-1.5,1);	\coordinate (p1324) at (0,1);	\coordinate (p2134) at (1.5,1);
															\coordinate (p1234) at (0,0);
\draw[very thick] (p1243)--(p1234) (p1324)--(p1234) (p1342)--(p1243) (p1423)--(p1324) (p1432)--(p1423) (p1432)--(p1342) (p2134)--(p1234) (p2143)--(p1243) (p2143)--(p2134) (p2314)--(p1324) (p2341)--(p1342) (p2413)--(p1423) (p2413)--(p2314) (p2431)--(p1432) (p2431)--(p2341) (p3124)--(p2134) (p3142)--(p2143) (p3214)--(p3124) (p3214)--(p2314) (p3241)--(p3142) (p3241)--(p2341) (p3412)--(p2413) (p3421)--(p3412) (p3421)--(p2431) (p4123)--(p3124) (p4132)--(p4123) (p4132)--(p3142) (p4213)--(p4123) (p4213)--(p3214) (p4231)--(p4132) (p4231)--(p3241) (p4312)--(p4213) (p4312)--(p3412) (p4321)--(p4312) (p4321)--(p4231) (p4321)--(p3421);
\draw[thick,blue] (p1243)--(p1423) (p1324)--(p3124) (p1342)--(p3142) (p1423)--(p4123) (p1432)--(p3412) (p1432)--(p4132) (p2143)--(p4123) (p2413)--(p4213) (p2431)--(p4231) (p4132)--(p4312);
\draw[red] (p1324)--(p1342) (p2134)--(p2314) (p2143)--(p2341) (p2143)--(p2413) (p2314)--(p2341) (p2413)--(p2431) (p3124)--(p3142) (p3142)--(p3412) (p3214)--(p3241) (p3214)--(p3412) (p3241)--(p3421) (p4213)--(p4231);
\foreach \name/\loc in {1234/p1234, 1243/p1243, 1324/p1324, 1342/p1342, 1423/p1423, 1432/p1432, 2134/p2134, 2143/p2143, 2314/p2314, 2341/p2341, 2413/p2413, 2431/p2431, 3124/p3124, 3142/p3142, 3214/p3214, 3241/p3241, 3412/p3412, 3421/p3421, 4123/p4123, 4132/p4132, 4213/p4213, 4231/p4231, 4312/p4312, 4321/p4321} \node at (\loc) [rectangle,draw,fill=white] {\sf\scriptsize\name};
\draw[very thick] (5,3.5)--(6,3.5); \node at (7,3.5) {$\geq_B,\ \geq_P,\ \geq_W$};
\draw[thick, blue] (5,3)--(6,3); \node at (6.7,3) {$\geq_B,\ \geq_P$};
\draw[red] (5,2.5)--(6,2.5); \node at (6.44,2.5) {$\geq_B$};
\end{tikzpicture}
\caption{Bruhat, pseudoreachability, and left weak order on $\Sym_4$\label{fig:s4}}
\end{center}
\end{figure}

\section{Pseudoreachability order and bubble-sorting order}\label{sorting}

The theory of normal forms in a Coxeter system was introduced by du~Cloux~\cite{duCloux} and is described in~\cite[\S3.4]{BB}.  We sketch here the facts we will need; see especially~\cite[Example 3.4.3]{BB}, which describes normal forms in the symmetric group in terms of bubble-sorting.  Let $\sigma_k=s_1\cdots s_k$ and $\omega_n=\sigma_{n-1}\cdots\sigma_1$; then $\omega_n$ is a reduced word for $w_0\in\Sym_n$.  Every $x\in\Sym_n$ has a unique \textbf{conormal form}: a reduced word $N(w)$ of the form $v_{n-1} v_{n-2} \cdots v_2 v_1$, where $v_k=s_j s_{j+1}\cdots s_k$ is a suffix of $\sigma_k$.  The conormal form is the reverse of the lexicographically first reduced word for $x^{-1}$ (that is, of the normal form of $x^{-1}$, as described in~\cite{BB}).  Thus $x$ is characterized by the sequence
\[\lambda(x)=(\lambda_{n-1}(x),\dots,\lambda_1(x))=(|v_{n-1}|,\dots,|v_1|)\in[0,n-1]\x[0,n-2]\x\cdots\x[0,1].\]

Armstrong~\cite{Armstrong} defined a general class of \emph{sorting orders} on a Coxeter system $(W,S)$: one fixes $w\in W$ and chooses a reduced word $\omega$ (the ``sorting word'') for $w\in W$, then partially orders all group elements expressible as a subword of~$\omega$ by inclusion between their lexicographically first such expressions.  Armstrong proved that for every reduced word for the top element of a finite Coxeter group, the sorting order is a distributive lattice intermediate between the weak and Bruhat orders.  In the case that $W=\Sym_n$ and $\omega=\omega_n$, the sorting order is equivalent to comparing $\lambda(x)$ and $\lambda(y)$ componentwise, hence is isomorphic to $C_2\x\cdots\x C_n$, where $C_i$ denotes a chain with $i$ elements.

\begin{prop} \label{v-reduction}
Let $x,y\in\Sym_n$ with $x(n)=y(n)=n$, and let $v=s_j s_{j+1} \cdots s_{n-1}$ be a suffix of $s_1 \; \cdots \; s_{n-1}$.
Then $x\reach y$ if and only $vx\reach vy$.
\end{prop}

\begin{proof}
If $v=e$, there is nothing to prove.  Otherwise, by~\eqref{RC}, it suffices to show that for every $i\in[n]$, we have
\begin{subequations}
\begin{equation} \label{RC-for-xy}
[y(i),x(i)]\subseteq y[i,n]
\end{equation}
if and only if
\begin{equation} \label{RC-for-v}
[(vy)(i),(vx)(i)]\subseteq vy[i,n].
\end{equation}
\end{subequations}
This is clear if $i=n$, so we assume henceforth that $i\neq n$.  Moreover,
\[v(k)=\begin{cases}
k &\text{ if } k<j,\\
k+1 &\text{ if } j\leq k<n,\\
j &\text{ if } k=n\end{cases}
\qquad\text{and}\qquad
v^{-1}(k)=\begin{cases}
k &\text{ if } k<j,\\
n &\text{ if } k=j,\\
k-1 &\text{ if } k>j.\end{cases}
\]
In particular, if $i\neq n$, then $x(i)>y(i)$ if and only if $v(x(i))>v(y(i))$.  We assume henceforth that these two equivalent conditions hold, since if both fail then~\eqref{RC-for-xy} and~\eqref{RC-for-v} are both trivially true.
The proofs of the two directions now proceed very similarly.
\medskip

$\eqref{RC-for-xy}\implies\eqref{RC-for-v}$:\quad There are three cases.

\emph{Case 1a: $j>x(i)$.}  Then $v$ fixes $[1,x(i)]$ pointwise, so $[(vy)(i),(vx)(i)]=v[y(i),x(i)]\subseteq vy[i,n]$ (applying $v$ to both sides of~\eqref{RC-for-xy}).

\emph{Case 1b: $y(i) < j \leq x(i)$.}
Then $(v x)(i) = x(i) + 1$ and $(v y)(i) = y(i)$, so
\begin{align*}
[(vy)(i),(vx)(i)] &= [y(i),j-1]\cup\{j\}\cup[j+1,x(i)+1]\\
&= v[y(i),j-1]\cup\{v(n)\}\cup v[j,x(i)]\\
&= v\left( [y(i),x(i)]\cup \{y(n)\}\right)\\
&\subseteq vy[i,n]
\end{align*}
establishing~\eqref{RC-for-v}.

\emph{Case 1c: $j \leq y(i)$.}  Similarly to Case~1a, we have $[(vy)(i),(vx)(i)] = [y(i)+1,x(i)+1] = v[y(i),x(i)]\subseteq vy[i,n]$, as desired.

$\eqref{RC-for-v}\implies\eqref{RC-for-xy}$:\quad
Applying $v^{-1}$ to both sides of~\eqref{RC-for-v} gives
$v^{-1}[vy(i),vx(i)]\subseteq y[i,n]$,
so in order to prove~\eqref{RC-for-xy} It is enough to show that
\begin{equation} \label{enough}
[y(i),x(i)]\subseteq v^{-1}[vy(i),vx(i)]
\end{equation}

Moreover, the earlier assumption $i\neq n$ implies that $vx(i)\neq j$ and $vy(i)\neq j$.

\emph{Case 2a: $j > vx(i)$.} Then $v^{-1}$ fixes the set $[1,vx(i)]$ pointwise, so in particular $[y(i),x(i)] = [vy(i),vx(i)] = v^{-1}[vy(i),vx(i)]$, establishing~\eqref{enough}.

\emph{Case 2b: $vy(i) < j < vx(i)$.} Then $y(i) = vy(i)$ and $x(i) = vx(i)-1$, so
\begin{align*}
[y(i),x(i)] &= [vy(i),j-1] \cup [j,vx(i)-1] \\
&= v^{-1}[vy(i),vy(n)-1] \cup v^{-1}[vy(n)+1,vx(i)] \\
&\subseteq v^{-1}[vy(i),vx(i)].
\end{align*}

\emph{Case 2c: $j < vy(i)$.} Then $[y(i),x(i)] = [vy(i)-1,vx(i)-1] = v^{-1}[vy(i),vx(i)]$, again implying~\eqref{enough}.
\end{proof}

\begin{theorem} \label{thm:same}
The pseudoreachability order coincides with the bubble-sorting order.
\end{theorem}

\begin{proof}
It suffices to show that the two partial orders have the same covering relations, i.e., that
\[x\gtrdot_P y \quad\iff\quad \lambda(x)\gtrdot_C\lambda(y).\]

We induct on $n$; the base case $n=1$ is trivial.  Let $x,y\in\Sym_n$ with $n>1$, and let their conormal forms be
\[
x = u\bar x = (s_i\cdots s_{n-1}) \bar x,\qquad
y = v\bar y = (s_j\cdots s_{n-1}) \bar y
\]
where $i=x(n)=n-\lambda_{n-1}(x)$ and $j=y(n)=n-\lambda_{n-1}(y)$.
\medskip

($\impliedby$)\quad Suppose that $\lambda(x)\gtrdot_C\lambda(y)$.  Then either $i=j-1$ or $i=j$.  If $i=j-1$, then $\lambda(\bar x)=\lambda(\bar y)$, so $\bar x=\bar y$ and $x=s_iy$, which by Prop~\ref{lem:inv-1} implies $x\gtrdot_P y$.  If $i=j$, then $\lambda(\bar x)\gtrdot_C\lambda(\bar y)$.  Then $\bar x\gtrdot_P\bar y$ by induction, so $v\bar x\gtrdot_P v\bar y=y$
by Prop.~\ref{v-reduction}.
\medskip

($\implies$)\quad Suppose that $x\gtrdot_Py$. Then $x \gtrdot_B y$ by Theorem~\ref{thm:bruhat}, so $i\leq j$ (as noted in the proof of Prop.~\ref{cf-df-facts}).

If $i<j$, then $v$ is a proper suffix of $u$. By the definition of Bruhat order it must be the case that $x=yt_{a,b}$ for some $a<b$; in fact $b=n$ (otherwise $x(n)=y(n)$).  Then $x(n)=y(a)$ and $x(a)=y(n)$, and $x(k)=y(k)$ for $k\not\in\{a,n\}$.  Moreover, $y(a)<x(a)$ (since $x \gtrdot_B y$ and not vice versa).  On the other hand, if $y(a)\leq x(a)-2$, so that $y(a)<c<x(a)=y(n)$ for some $c$, then by~\eqref{RC} $c=y(k)$ for some $k\in[a+1,n-1]$, and in particular $x$ has at least three more inversions than $y$ --- not only $(a,n)$, but also $(a,k)$ and $(k,n)$, which contradicts the assumption $x\gtrdot_Py$.  Therefore $y(a)=x(a)-1$, i.e., $x(n)=y(n)-1$.  We conclude that $x=s_iy$, so $\lambda(x)\gtrdot_C\lambda(y)$ using the conormal forms above.

If $i=j$, then $u=v$, so $\bar x\gtrdot_P\bar y$ by Prop.~\ref{v-reduction}.  By induction $\lambda(\bar x)\gtrdot_C\lambda(\bar y)$, and prepending $n-i$ gives
$\lambda(x)\gtrdot_C\lambda(y)$ as well.
\end{proof}

\section{Pattern avoidance and reachability} \label{sec:avoid}

In this section, we establish two sufficient conditions for reachability using pattern avoidance. (It is dubious whether pattern avoidance conditions can completely characterize reachability.)

Let $\pi\in\Sym_n$ and $\sigma\in\Sym_m$, where $m\leq n$.  A \defterm{$\sigma$-pattern} is a subsequence $\pi(i_1),\dots,\pi(i_m)$ in the same relative order as $\sigma$, i.e., such that $1\leq i_1<\cdots<i_m\leq n$ and $\pi(i_j)<\pi(i_k)$ if and only if $\sigma(j)<\sigma(k)$.  If $\pi$ contains no $\sigma$-pattern then we say that $\pi$ \defterm{avoids} $\sigma$.

\begin{theorem} \label{thm:213-avoiding}
If $x\geq_By$ and $y$ avoids $213$, then $x\reach y$.
\end{theorem}

\begin{proof}
Suppose that $x\geq_By$ and $y$ avoids $213$, but $x\notreach y$.
Let $i$ be any index such that $\df_i(x,y) = 0$.  By Prop.~\ref{cf-df-facts} we know that $1<i<n$ and that $y(i) < x(i)$.  In particular, $m\neq i$, where $m = y^{-1}(x(i))$; that is, $y(m)=x(i)$.

First, suppose that $m > i$.  We claim that there exists some $u<i$ such that $y(i) < y(u) < y(m)$.  Otherwise, $J_i\geq y(m)-y(i)+1$, and then $k=y(m)-y(i)$ has the properties $k<J_i$ and $y(i)+k=y(m)=x(i)$, so $k\in\DF_i(x,y)$, contradicting the assumption $\df_i(x,y) = 0$.  Therefore $y(u), y(i), y(m)$ is a 213-pattern.

Second, suppose that $m < i$. If $y(k) > y(m)$ for some $k>i$, then $y(m),y(i),y(k)$ is a 213-pattern. On the other hand, suppose that $y(k) < y(m) = x(i)$ for all $k > i$ (hence for all $k \ge i$).
Then
\[\{k\in[i,n]:\ y(k)<x(i)\}=[i,n]~\supsetneq~[i+1,n]\supseteq\{k\in[i,n]:\ x(k)<x(i)\}\subseteq[i+1,n]\]
so
\begin{align*}
\#\{k\in[i,n]:\ y(k)<x(i)\} &> \#\{k\in[i,n]:\ x(k)<x(i)\}\\
\therefore\quad \#\{k\in[1,i-1]:\ y(k)<x(i)\} &< \#\{k\in[1,i-1]:\ x(k)<x(i)\}\\
\therefore\quad \#\{k\in[1,i-1]:\ y(k)\geq x(i)\} &> \#\{k\in[1,i-1]:\ x(k)\geq x(i)\}.
\end{align*}
That is, $y\langle i-1,x(i)\rangle>x\langle i-1,x(i)\rangle$, contradicting the assumption $x\geq_By$.
\end{proof}

Theorem~\ref{thm:213-avoiding} partially answers the question of when the converse of Theorem~\ref{thm:bruhat} holds, i.e., which Bruhat relations are also relations in pseudoreachability order. We next study if there is an analogous condition on $x$, rather than $y$, that suffices for reachability.  One such condition that allows us to restrict $x$ instead of $y$ is to ensure that only very few entries $x(i)$ are large with respect to $i$.

\begin{lemma} \label{lem:x1}
Let $x\in\Sym_n$.  The following conditions are equivalent:
\begin{enumerate}
\item $x^{-1}(i) \leq i+1$ for all $i\in[n]$.
\item $x\langle j,j\rangle = 1$ for all $j \in [n]$.
\item $x$ avoids both 231 and 321.
\item $x$ is of the form $s_{i_1}\cdots s_{i_k}$, where $n-1\geq i_1>\cdots>i_k\geq1$.
\end{enumerate}
\end{lemma}
The number of these permutations is $2^{n-1}$, which is easiest to see from condition~(4).
Conditions~(1) and~(3) were mentioned respectively by J.~Arndt (June 24, 2009) and M.~Riehl (August 5, 2014) respectively in the comments on sequence A000079 in \cite{OEIS}.
Accordingly, we will call a permutation satisfying the condition of Lemma~\ref{lem:x1} an \defterm{AR permutation} (for Arndt--Riehl).


\begin{proof}
$(1)\iff(2)$: Formula~\eqref{angle-brackets} implies that
\begin{align*}
\forall j\in[n]:\ x\langle j,j\rangle = 1
&\iff \forall j\in[n]:\ [1,j-1]\subseteq x[1,j]\\
&\iff \forall j\in[n]:\ x^{-1}[1,j-1]\subseteq [1,j]\\
&\iff \forall i\in[n]:\ x^{-1}[1,i]\subseteq [1,i+1]
\end{align*}
since the last two statements differ only by the trivially true cases $i=0$ and $i=n$.

$(3)\iff(1)$: Condition~(3) holds if and only if no digit $i\in[n]$ occurs later than position $i+1$, but this is precisely condition (1).

$(4)\iff(1)/(3)$: Let $Y_n$ be the set of permutations in $\Sym_n$ satisfying the equivalent conditions~(1) and~(3), and let $Z_n$ be the set satisfying condition (4).  For $n\leq2$ we evidently have $Y_n=Z_n=\Sym_n$.  For $n\geq 3$, we proceed by induction.  Observe that $Z_n=Z_{n-1}\cup s_{n-1}Z_{n-1}$, and that left-multiplication by $s_{n-1}$ (i.e., swapping the locations of $n-1$ and $n$) does not affect condition (1), which is always true for $i\in\{n-1,n\}$.  Therefore $Z_n\subseteq Y_n$.

On the other hand, if $w\in Y_n$ then $w_n\in\{n-1,n\}$, otherwise $w_n$, together with the digits $n-1$ and $n$, would form a 231- or 321-pattern.  Therefore, $w'_n=n$, where either $w'=w$ or $w'=s_{n-1}w$.  By induction $w'\in Z_{n-1}$, so $w\in Z_n$ as desired.
\end{proof}

\begin{corollary} \label{cor:arn-bru}
If $x$ is AR and $y\leq_Bx$, then $y$ is AR as well.
\end{corollary}

\begin{proof}
Lemma \ref{lem:x1} asserts that $x\langle i,i\rangle = 1$ for all $i \in [n]$. Since $y\leq_Bx$, $y\langle i,i\rangle = 1$ or $0$, but the latter could not happen by the pigeonhole principle.
\end{proof}

An \defterm{exceedance} of a permutation $x\in\Sym_n$ is an index $k\in[n]$ such that $x(k)>k$.

\begin{lemma} \label{lem:x2}
Let $x \in \Sym_{n}$ be an AR permutation.  Suppose that $k$ is an exceedance of $x$, and let $i=x(k)$.  Then $x(j)=j-1$ for all $j\in[k+1,i]$.
\end{lemma}

\begin{proof}
The argument of Lemma~\ref{lem:x1} implies that $[1,k-1]\subseteq x[1,k]$; however, since $x(k)>k$ we have in fact
$[1,k-1]=x[1,k-1]$.

Now let $j\in[k+1,i]$.  Lemma~\ref{lem:x1} also asserts that $x\langle j,j\rangle=\#A_j=1$, where $A_j=\{m\in[j]:\ x(m)\geq j\}$.  Certainly $k\in A_j$, so $j\not\in A_j$, that is, $x(j)<j$.  But since $x(j)\geq k$ for each such $j$, we can infer in turn that $x(k+1)=k$, $x(k+2)=k+1$, \dots, $x(i)=i-1$.
\end{proof}

\begin{theorem} \label{thm:x}
If $x\geq_By$ and $x$ is AR, then $x\reach y$.
\end{theorem}

\begin{proof}
Suppose that $x\geq_By$ and $x$ is AR, but $x\notreach y$.
Let $i$ be some index such that $\df_{i}(x, y) = 0$.
By~\eqref{RC}, there exists $j < i$ such that
\begin{equation} \label{banana}
y(i) < y(j) \le x(i).
\end{equation}
By Lemma \ref{lem:x1}, $x\langle i,i\rangle = 1$; that is, there exists some (unique) $k \leq i$ such that $x(k) \ge i$.

First, suppose that $k=i$.  Then $x\langle i-1,i\rangle = 0$, and $y\langle i-1,i\rangle = 0$ as well because $y \leq_B x$.  Hence $y[1,i-1]=[i-1]$.  But then~\eqref{banana} implies that $y(i)<y(j)\leq i-1$ as well, a contradiction.

Second, suppose that $k < i$.  Then $y(i)<x(i)<i$ by Lemma~\ref{lem:x2}, so $y(i)\leq i-2$.
Set $p=y(i)$; then $y^{-1}(p)=i\geq k+2$.  But then $y$ is not AR, which violates Corollary~\ref{cor:arn-bru}.
\end{proof}

\bibliographystyle{plain}
\bibliography{biblio}
\end{document}